\documentclass[12pt]{amsart}

\pdfoutput=1
\usepackage{latexsym}
\usepackage[centertags]{amsmath}
\usepackage{amsfonts}
\usepackage{amssymb}
\usepackage{amsthm}
\usepackage{newlfont}
\usepackage{graphics}
\usepackage{color}

\usepackage[demo]{graphicx} 
\usepackage{wrapfig,floatrow}
\usepackage{lipsum} 
\usepackage[font=small,labelfont=bf,margin=5mm]{caption}

\usepackage[usenames,dvipsnames]{xcolor}
\usepackage{graphicx}

\newtheorem{theo}{Theorem}

\newtheorem{lem} [theo]{Lemma}
\newtheorem{cor}[theo]{Corollary}
\newtheorem{prop}[theo]{Proposition}

\makeatletter \@addtoreset{equation}{section}
\@addtoreset{theo}{}\makeatother

\def\Oge{\mathop{\Omega}_{\geq}}
\def\Oeq{\mathop{\Omega}_{\mu=}}

\def\Z{\mathbb{Z}}

\newcommand{\area}{\operatorname{\texttt{area}}}

\newcommand{\dinv}{\operatorname{\texttt{dinv}}}
\newcommand{\bounce}{\operatorname{\texttt{bounce}}}

\def\k{\vec{k}}

\def \CD {{\cal D}}

\def\oD{\overline{D}}

\title[The $q,t$-symmetry for $C_{(k_1,k_2,k_3)}(q,t)$ and $C_{(k,k,k,k)}(q,t)$]{The $q,t$-symmetry of the generalized $q,t$-Catalan number $C_{(k_1,k_2,k_3)}(q,t)$ and $C_{(k,k,k,k)}(q,t)$}

\author{Guoce Xin$^{1}$ and Yingrui Zhang$^{2,*}$}

 \address{ $^{1,2}$School of Mathematical Sciences, Capital Normal University,
 Beijing 100048, PR China}
 \email{$^1$\texttt{guoce\_xin@163.com}\ \& $^2$\texttt{zyrzuhe@126.com}}
\date{May 24, 2022}
\thanks{This work was partially supported by NSFC(12071311).}

\begin{document}

\begin{abstract}
We give two proofs of the $q,t$-symmetry of the generalized $q,t$-Catalan number $C_{\vec{k}}(q,t)$ for $\vec{k}=(k_1,k_2,k_3)$. One is by using MacMahon's partition analysis as we proposed; the other is
a direct bijection. We also prove $C_{(k,k,k,k)}(q,t) = C_{(k,k,k,k)}(t,q)$ by using MacMahon's partition analysis.
\end{abstract}

\maketitle

\noindent
\begin{small}
 \emph{Mathematic subject classification}: Primary 05A19; Secondary 05E99.
\end{small}

\noindent
\begin{small}
\end{small}

\section{Introduction}
In their study of the space $\mathcal{DH}_n$ of diagonal harmonics \cite{qt-Catalan}, Garsia and Haiman introduced a $q, t$-analogue of the Catalan numbers,
which they called the $q, t$-Catalan sequence. There are several equivalent characterizations of the (original) $q, t$-Catalan sequence, which includes two combinatorial
formulas:
One is Haiman's dinv-area $q,t$-Catalan sequence; the other is Haglund's area-bounce $q,t$-Catalan sequence. See \cite{Haglund-bounce,Higher-qt-Catalan}.
Further information about $q,t$-Catalan sequence and related results can be found in \cite{Garsia-Haglund-proof-positivity,Gorsky-Mazin,Hag-book08}.
In \cite{Xin-Zhang2}, we introduced $q,t$-Catalan numbers $C_{\lambda}(q,t)$ of type (a partition) $\lambda$ as an extension of Haglund and Haiman's combinatorial
formula for ordinary $q,t$-Catalan numbers.
We also investigated the $q,t$-symmetry of $C_{\lambda}(q,t)$:
The symmetry is easily proved when the length of $\lambda$ is $\ell(\lambda)=2$; The symmetry  may be proved by using MacMahon's partition analysis technique when $\ell(\lambda)=3$;
No symmetry holds in general when $\ell(\lambda)\ge 4$, but we conjecture the $q,t$-symmetry of $C_{\lambda}(q,t)$ when
$\lambda=((a+1)^s,a^{n-s})$.

We find it better to define $C_{\k}(q,t)$ for any ordered partition, i.e., vector $\k=(k_1,\dots, k_n)$ of positive integers, by
\begin{align*}
  C_{\k}(q,t)=\sum_{ D \in \CD_{\k}} q^{\area(D)} t^{\bounce(D)},
\end{align*}
where the sum ranges over all $\k$-Dyck paths $D$, and $\area(D)$ and $\bounce(D)$ are two statistics of $D$.

In the special case when $k_i = k$ for all $i$, denote by $k^n = (k,\dots, k)\in \mathbb{Z}^n$ and
\begin{align*}
  C_{k^n}(q,t)=\sum_{ D \in \CD_{k^n}} q^{\area(D)} t^{\bounce(D)},
\end{align*}
where the sum ranges over all $k^n$-Dyck paths $D$.

Our main result is the following.
\begin{theo}\label{t-qtsym}
For $\k=(k_1,k_2,k_3)$, the $q,t$-Catalan number $C_{\k}(q,t)$ is $q,t$-symmetric, i.e., $C_{\k}(q,t)=C_{\k}(t,q)$.
\end{theo}

A direct corollary is the following, since
$$C_\lambda(q,t) = \sum_{\lambda(\k)=\lambda} C_{\k}(q,t),$$
where $\lambda(\k)$ is the partition obtained by arranging the entries of $\k$ decreasingly.
\begin{cor}
For any partition $\lambda$ of length $3$, the $q,t$-Catalan number $C_{\lambda}(q,t)$ is $q,t$-symmetric, i.e., $C_{\lambda}(q,t)=C_{\lambda}(t,q)$.
\end{cor}

Note that one can define
$$H_{\k}(q,t)=\sum_{ D \in \CD_{\k}} q^{\dinv(D)} t^{\area(D)}, \quad  \text{ and }  H_{\lambda}=\sum_{\lambda(\k)=\lambda} H_{\k}(q,t).$$
We have $C_\lambda(q,t)=H_\lambda(q,t)$ by the fact that the sweep map takes dinv to area, and area to bounce, but
$C_{\k}(q,t)\neq H_{\k}(q,t)$ because the sweep map takes a path in $\CD_{\k}$ to a path in $\CD_{\k'}$ for some $k'$ with $\lambda(\k)=\lambda(\k')$.
Indeed, $H_{\k}(q,t)$ is not symmetric even when $\k$ is of length $3$. In this sense, the bounce statistic is nicer than the dinv statistic.

We give two proofs of Theorem $\ref{t-qtsym}$. One is by MacMahon's partition analysis as we proposed; the other is by
a direct bijection. For $\k=(k_1,k_2,k_3)$, the statistics area and bounce have a simple description.
Dyck paths $D\in \CD_{\k}$ are uniquely determined by their red ranks $(r_1=0,r_2,r_3)$, which satisfy the conditions $0\le r_2\le k_1$ and $0\le r_3\le r_2+k_2$.
Let $\lceil \alpha \rceil$ denote the least integer greater than or equal to $\alpha$.
The area of $D$ is simply $\area(D)=r_2+r_3$, and the bounce of $D$ is given by
\begin{align*}
\bounce(D)=\left\{
               \begin{array}{ll}
                 2(k_1-r_2)+r_2+k_2-r_3-\min(r_2,k_2), & \text{ if } r_2+k_2-r_3 \geq 2\min(r_2,k_2) ; \\
                 2(k_1-r_2)+ \lceil {\frac{r_2+k_2-r_3} 2}\rceil, & \text{otherwise.}
               \end{array}
             \right.
\end{align*}

The theorem can be proved by MacMahon's partition analysis, which was developed by MacMahon for solving problems
in partition theory. The fundamental ingredient is MacMahon's Omega operator on the $\lambda$ variables defined by
\begin{align*}
\Oge \sum_{i_1,i_2,\dots i_n\in \Z} A_{i_1,\dots, i_n} \lambda_1^{i_1}\lambda_2^{i_2}\cdots \lambda_{n}^{i_n} &= \sum_{i_1,i_2,\dots i_n\ge 0} A_{i_1,\dots, i_n},\\
\Oeq \sum_{i_1,i_2,\dots i_n\in \Z} A_{i_1,\dots, i_n} \mu_1^{i_1}\mu_2^{i_2}\cdots \mu_{n}^{i_n} &= A_{0,\dots, 0}.
\end{align*}
In other words, the $\Oge$ operator extracts all terms with nonnegative power in $\lambda_1,\dots, \lambda_n$ and then set them to be equal to $1$,
and the $\Oeq$ operator extracts the term independent of the $\mu$ variables. The to-be-eliminated variables $\lambda_i$ or $\mu_j$ are
usually clear from the context.
Andrews et. al. developed the Mathematica package Omega \cite{Ome-package} to eliminate the $\lambda$ variables. We use Xin's Ell2 Maple package \cite{fastCT} in our computation.

The paper is organized as follows. In this introduction, we have introduced the basic concepts.
In Section \ref{s-Mac-proof} we prove Theorem \ref{t-qtsym} by MacMahon's partition analysis.
In Section \ref{s-bije-proof}, we give a bijective proof of Theorem \ref{t-qtsym}.
In Section \ref{s-Mac-proof2} we prove $ C_{k^4}(q,t) =  C_{k^4}(t,q)$  by MacMahon's partition analysis.

\section{An Algebraic Proof of Theorem \ref{t-qtsym}\label{s-Mac-proof}}

\subsection{Basic Idea}
For $\k=(k_1,k_2,k_3)$, Dyck paths $D\in \CD_{\k}$ are uniquely determined by their red ranks $(r_1=0,r_2,r_3)$,
which satisfy the conditions $0\le r_2\le k_1$ and $0\le r_3\le r_2+k_2$.

The area of $D$ is simply given by $
\area(D)=r_2+r_3.
$
We also have an explicit formula of $\bounce(D)=b(r_2,r_3)$ as follows.
$$b(r_2,r_3)=\left\{
               \begin{array}{ll}
                 2(k_1-r_2)+r_2+k_2-r_3-\min(r_2,k_2), & \text{ if } r_2+k_2-r_3 \geq 2\min(r_2,k_2) ; \\
                 2(k_1-r_2)+ \lceil {\frac{r_2+k_2-r_3} 2}\rceil, & \text{otherwise.}
               \end{array}
             \right.
$$
This formula is piecewise linear in $r_2$ and $r_3$.

We prove Theorem $\ref{t-qtsym}$ by MacMahon's partition analysis technique, which applies to
sum over linear constraints. Indeed, we can construct the generating function with respect to $k_1,k_2,k_3, r_2,r_3$:
\begin{align*}
F(x_1,x_2,x_3,y_2,y_3,q,t) &= \sum_{k_1,k_2,k_3\ge 0}x_1^{k_1}x_2^{k_2}x_3^{k_3} \sum_{D\in \CD_{(k_1,k_2,k_3)}} q^{\area(D)} t^{\bounce(D)} y_2^{r_2}y_3^{r_3}  \\
                       &= \sum_{k_1,k_2,k_3\ge 0}x_1^{k_1}x_2^{k_2}x_3^{k_3} \sum_{0\le r_2\le k_1, \ 0\le r_3\le r_2+k_2   } q^{r_2+r_3} t^{b(r_2,r_3)} y_2^{r_2}y_3^{r_3},
\end{align*}
where we allowed $k_i=0$ for $i=1,2,3$, which does not affect the computation.

Then it is sufficient to prove the $q,t$ symmetry of
$$F(x_1,x_2,x_3,1,1,q,t)= \sum_{k_1,k_2,k_3\ge 0}x_1^{k_1}x_2^{k_2}x_3^{k_3} C_{(k_1,k_2,k_3)}(q,t).$$

We need to use the following Lemma to simplify our computation.
\begin{lem}\label{skillformula}
Let $A$ and $z$ be integers with $z\ge 2$. Write $p=\lceil {\frac{A} z}\rceil$, so $ 0\le pz-A\le z-1$. Then
$$\chi(zp \in \{A+i|i=0\dots z-1\})=\mathop{\Omega}_{\mu=} \sum_{i=0}^{z-1}\mu^{A-zp+i}=\mathop{\Omega}_{\mu=}(\sum_{i=0}^{z-1}\mu^i)\mu^{A-zp}$$
where $\chi(true)=1$ and $\chi(false)=0$.
\end{lem}

\begin{proof}
It's trivial.
\end{proof}

\subsection{Crude generating function}
Due to the piecewise linearity of $b(r_2,r_3)$, we divide the generating function $F(x_1,x_2,x_3,y_2,y_3,q,t)$ into two parts,
each having two cases. It is convenient to write the generating function $F(x_1,x_2,x_3,y_2,y_3,q,t)=F=F_1+F_2=F_{11}+F_{12}+F_{21}+F_{22}$.

Part 1: $r_2>k_2$ (written as $r_2-k_2-1 \geq 0$), which implies that $\min(r_2, k_2) = k_2$. The corresponding generating function is $F_1$, and we divide it into $F_1=F_{11}+F_{12}$
by considering the following two cases:

   Case 1: $r_2+k_2-r_3 \geq 2k_2$, i.e., $r_2-r_3-k_2 \ge 0$. Then, we have $$b(r_2,r_3) = 2(k_1-r_2)+r_2+k_2-r_3-k_2 = 2k_1-r_2-r_3.$$
   So we have
\begin{align*}
F_{11}=&\!\!\!\!\!\sum_{k_1,k_2,k_3\ge 0}x_1^{k_1}x_2^{k_2}x_3^{k_3} \sum_{0\le r_2\le k_1, \ 0\le r_3\le r_2+k_2,  \atop  r_2-k_2-1 \geq 0, \  r_2-r_3-k_2\ge 0} q^{r_2+r_3} t^{2k_1-r_2-r_3} y_2^{r_2}y_3^{r_3}\\
=&\!\!\!\!\!\!\!\!\!\!\!\!\sum_{k_1,k_2,k_3,r_2,r_3\ge 0} \!\!\Oge \left( {x_{{1}}}^{k_{{1}}}{x_{{2}}}^{k_{{2}}}{x_{{3}}}^{k_{{3}}}{y_{{2}}}^{r_{
{2}}}{y_{{3}}}^{r_{{3}}}{q}^{r_{{2}}+r_{{3}}}{t}^{2\,k_{{1}}-r_{{2}}-r
_{{3}}}{\lambda_{{1}}}^{k_{{1}}-r_{{2}}}{\lambda_{{2}}}^{k_{{2}}+r_{{2
}}-r_{{3}}}{\lambda_{{3}}}^{r_{{2}}-k_{{2}}-1}{\lambda_{{4}}}^{r_{{2}}
-k_{{2}}-r_{{3}}} \right)\\
=&\Oge\frac {1}{{\lambda_{{3}}} \left( 1-x_{{1}}{t}^{2}\lambda_{{1}} \right) \left( 1-{\frac {x_{{2}}\lambda_{{2}}}{\lambda_{{3}}\lambda_{{4}}}}
 \right) \left( 1-{\frac {\lambda_{{2}}\lambda_{{3}}\lambda_{{4}
}qy_{{2}}}{t\lambda_{{1}}}} \right) \left( 1-{\frac {y_{{3}}q}{
\lambda_{{2}}\lambda_{{4}}t}} \right)(1-x_3)};
\end{align*}

  Case 2: $r_2+k_2-r_3 < 2k_2$, i.e., $k_2-r_2+r_3-1 \ge 0$. Then, we have
  \begin{align*}
  b(r_2,r_3)= 2(k_1-r_2)+ \lceil {\frac{r_2+k_2-r_3} 2}\rceil.
  \end{align*}

Let $p=\lceil {\frac{r_2+k_2-r_3} 2}\rceil$. Note that $p\ge 0$, because the red ranks $(r_1=0,r_2,r_3)$ of a Dyck path $D\in \CD_{\k}$ satisfy  $0\le r_3 \le r_2+k_2$.
By Lemma \ref{skillformula}, we can write the generating function $F_{12}$ as follows.
\begin{align*}
F_{12}=&\!\!\!\!\!\sum_{k_1,k_2,k_3\ge 0, }x_1^{k_1}x_2^{k_2}x_3^{k_3} \sum_{0\le r_2\le k_1, \ 0\le r_3\le r_2+k_2,\atop {k_2-r_2+r_3-1 \ge 0,\  r_2-k_2-1 \geq 0, \atop  p\ge 0,\ 2p=r_2+k_2-r_3 \text{ or } 2p=r_2+k_2-r_3+1  }} q^{r_2+r_3} t^{2(k_1-r_2)+p} y_2^{r_2}y_3^{r_3}\\
\\
=&\sum_{k_1,k_2,k_3,r_2,r_3,p\ge 0} \Oge \mathop{\Omega}_{\mu=} ({x_{{1}}}^{k_{{1}}}{x_{{2}}}^{k_{{2}}}{x_{{3}}}^{k_{{3}}}{y_{{2}}}^{r_{{2}}}{y_{{3}}}^{r_
{{3}}}{q}^{r_{{2}}+r_{{3}}}{t}^{2\,k_{{1}}-2\,r_{{2}}+p} \left( 1+\mu
 \right) {\mu}^{r_{{2}}+k_{{2}}-r_{{3}}-2\,p}\\
 &\ \ \ \ \ \ \ \ \ \ \ \ \ \ \ \ \ \ \ \ \ \ \ \ \ \ \ \ \ \ \ \ \ \ \ \ \ \ \ \ \ \ \ \ \ \ \ \ \ \ {\lambda_{{1}}}^{k_{{1}}-
r_{{2}}}{\lambda_{{2}}}^{k_{{2}}+r_{{2}}-r_{{3}}}{\lambda_{{3}}}^{r_{{
2}}-k_{{2}}-1}{\lambda_{{4}}}^{k_{{2}}+r_{{3}}-r_{{2}}-1})\\
\\
=&\Oge \mathop{\Omega}_{\mu=} \frac {1+\mu}{\lambda_3 \lambda_4 (1-x_1t^2\lambda_1)\left(1-\frac{x_2\mu\lambda_2\lambda_4}{\lambda_3}\right)
\left(1-\frac{y_2q\lambda_2\lambda_3\mu}{t^2\lambda_1\lambda_4}\right)\left(1-\frac{y_3q\lambda_4}{\lambda_2\mu}\right)
\left(1-\frac{t}{\mu^2}\right)(1-x_3)}.
\end{align*}

Part 2: $k_2\ge r_2$ (written as $k_2-r_2\ge 0$), which implies that $\min(r_2,k_2)=r_2$. The corresponding generating function is $F_2$.
By a similar computation for Part 1, we divide $F_2$ into $F_{21}+F_{22}$, whose formulas are given as follows.

Case 1: $r_2+k_2-r_3 \geq 2r_2$, i.e., $k_2-r_2-r_3 \ge 0$. Then, we have $b(r_2,r_3)=2k_1-2r_2+k_2-r_3$
and
\begin{align*}
F_{21}=
\Oge \frac{1}{(1-x_1t^2\lambda_1)(1-x_2t\lambda_2\lambda_3\lambda_4)
\left(1-\frac{y_2q\lambda}{t^2\lambda_1\lambda_3\lambda_4}\right)\left(1-\frac{y_3q}{t\lambda_2\lambda_4}\right)(1-x_3)}.
\end{align*}

Case 2: $r_2+k_2-r_3 < 2r_2$, i.e., $r_2+r_3-k_2-1 \ge 0$. Then, we have
  \begin{align*}
  b(r_2,r_3)= 2(k_1-r_2)+ \lceil {\frac{r_2+k_2-r_3} 2}\rceil
  \end{align*}
and
\begin{align*}
F_{22}=\Oge \mathop{\Omega}_{\mu=} \frac {1+\mu}{\lambda_4 (1-x_1t^2\lambda_1)\left(1-\frac{x_2\mu\lambda_2\lambda_3}{\lambda_4}\right)
\left(1-\frac{y_2q\lambda_2\lambda_4\mu}{t^2\lambda_1\lambda_3}\right)\left(1-\frac{y_3q\lambda_4}{\lambda_2\mu}\right)
\left(1-\frac{t}{\mu^2}\right)(1-x_3)}.
\end{align*}

\subsection{Obtain the generating function}
By using the maple package Ell2, we obtain:

\begin{align*}
F_{11}&=\frac{x_{{1}} y_{{2}} q \left({t+y_{{3}}q-{q}^{2}tx_{{1}}y_{{2}}y_{{3}}}\right)} {\left({1-x_{{1}}{t}^{2}}\right)\left({1-qy_{{2}}tx_{{1}}}\right)
\left({1-{q}^{2}y_{{2}}y_{{3}}x_{{1}}}\right)\left({1-x_{{2}}qy_{{2}}x_{{1}}t}\right)(1-x_3)},\\
F_{12}&=\frac{q ^{4}y_{{3}} ^{2}x_{{2}} y_{{2}} ^{2}x_{{1}} ^{2}\left({t+y_{{3}}q}\right)} {\left({1-x_{{1}}{t}^{2}}\right)\left({1-x_{{2}}qy_{{2}}x_{{1}}t}\right)
\left({1-{q}^{2}y_{{2}}y_{{3}}x_{{1}}}\right)\left({1-{y_{{3}}}^{2}{q}^{3}x_{{2}}y_{{2}}x_{{1}}}\right)(1-x_3)},\\
F_{21}&=\frac{1} {\left({1-tx_{{2}}}\right)\left({1-qx_{{2}}y_{{3}}}\right)\left({1-x_{{1}}{t}^{2}}\right)
\left({1-x_{{2}}qy_{{2}}x_{{1}}t}\right)(1-x_3)},\\
F_{22}&=\frac{y_{{3}} x_{{2}} y_{{2}} q ^{2}x_{{1}} \left({t+y_{{3}}q}\right)} {\left({1-x_{{1}}{t}^{2}}\right)\left({1-qx_{{2}}y_{{3}}}\right)\left({1-x_{{2}}qy_{{2}}x_{{1}}t}\right)
\left({1-{y_{{3}}}^{2}{q}^{3}x_{{2}}y_{{2}}x_{{1}}}\right)(1-x_3)}.
\end{align*}

By adding the above four formulas and setting $y_2=y_3=1$, we obtain
\begin{align}
F(x_1&,x_2,x_3,1,1,q,t) \nonumber \\
=&\frac{\left({1-q{t}^{2}x_{{1}}x_{{2}}}\right)\left({1-{q}^{2}tx_{{1}}x_{{2}}}\right)} {\left({1-qx_{{2}}}\right)\left({1-tx_{{2}}}\right)\left({1-qtx_{{1}}}\right)\left({1-{t}^{2}x_{{1}}}\right)
\left({1-{q}^{2}x_{{1}}}\right)\left({1-qtx_{{1}}x_{{2}}}\right)(1-x_3)}
\nonumber\\
=&\frac 1 {\left(1-qtx_1\right)\left(1-qtx_1x_2\right)(1-x_3)} \left(  \frac 1 { 1-{t}^{2}x_{{1}}}+{\frac {qx_{{2}}}{1-qx_
{{2}}}} \right)  \left(  \frac 1 {1-{q}^{2}x_{{1}}}+{
\frac {tx_{{2}}}{1-tx_{{2}}}} \right).\label{qtsym-formual}
\end{align}
The $q,t$-symmetry clearly follows from \eqref{qtsym-formual}.

\section{A Bijective Proof of Theorem \ref{t-qtsym}\label{s-bije-proof}}

In this section, we use the following notation to describe a Dyck path $D$ to obtain a simpler formula of $\bounce(D)$.
For $\k=(a,c,e)$,
A $\k$-Dyck path $D$ can be uniquely determined by the parameters $b,d$ as follows. It starts with
a red arrow $S^{a}$ (i.e. up step $(1,a)$) followed by $b$ blue arrows $W$ (down step $(1,-1)$),
then a red arrow $S^{c}$ followed by $d$ blue arrows $W$, and a red arrow $S^{e}$ followed by $(a+c+e-b-d)$ blue arrows $W$.
Clearly, the red ranks of $D$ are $r_1=0, \ r_2=a-b $ and $ r_3=a-b+c-d$.
See \cite{dinv-area,Xin-Zhang} for detailed concepts,
which are irrelevant here.

The formula of $\area(D)$ is as follows
$$\area(D)=r_2+r_3=2a-2b+c-d.$$
We recall that the formula of $\bounce(D)=b(r_2,r_3)$, i.e.,
$$b(r_2,r_3)=\left\{
               \begin{array}{ll}
                 2(k_1-r_2)+r_2+k_2-r_3-\min(r_2,k_2), & \text{ if } r_2+k_2-r_3 \geq 2\min(r_2,k_2) ; \\
                 2(k_1-r_2)+ \lceil {\frac{r_2+k_2-r_3} 2}\rceil, & \text{otherwise.}
               \end{array}
             \right.
$$
Then, we have a simpler formula of $\bounce(D)=B(b,d)$ as follows:
$$B(b,d)=\left\{
               \begin{array}{ll}
                 2b+d-\min(a-b,c), & \text{ if } d \geq 2\min(a-b,c) ; \\
                 2b+ \lceil {\frac{d} 2}\rceil, & \text{otherwise.}
               \end{array}
             \right.
$$

We will give an involution (a bijection whose square is the identity map) on $\CD_{\k}$ for $\k=(a,c,e)$ which interchanges $\area$ and $\bounce$.
If $a \leq c$, we have $\min(a-b,c)= a-b$ and the formula of
$B(b,d)$ becomes simpler. 
So we divide the involution into two sub-sections: one for the case $a\le c$ and the other for the case $a>c$.

Now we need to give the following Lemma.
\begin{lem}\label{lem-ymula}
Giving two non-negative integers $c$ and $d$. If the following three conditions hold
\begin{enumerate}
  \item $Y=\chi((c+\lceil {\frac {d} 2}\rceil) ~is ~odd)$;
  \item $d'=2\lfloor {\frac{d} 2}\rfloor+Y$;
  \item $Y'=\chi((c+\lceil {\frac {d'} 2}\rceil) ~is ~odd)$.
\end{enumerate}
Then $Y'=\chi(d ~is ~odd)$.
\end{lem}

\begin{proof}
We have $$c+\lceil {\frac {d'} 2}\rceil=c+\lfloor {\frac{d} 2}\rfloor+Y$$
and
$$c+\lfloor {\frac{d} 2}\rfloor+Y=\left\{
  \begin{array}{ll}
    c+\lfloor {\frac{d} 2}\rfloor+1=c+\lceil {\frac {d} 2}\rceil, & \hbox{ if } c+\lceil {\frac {d} 2}\rceil \text{is odd and } d \text{ is odd}; \\
   c+\lfloor {\frac{d} 2}\rfloor+1=c+\lceil {\frac {d} 2}\rceil+1, & \hbox{ if } c+\lceil {\frac {d} 2}\rceil \text{is odd and } d \text{ is even}; \\
    c+\lfloor {\frac{d} 2}\rfloor=c+\lceil {\frac {d} 2}\rceil-1,  & \hbox{ if } c+\lceil {\frac {d} 2}\rceil \text{is even and } d \text{ is odd}; \\
    c+\lfloor {\frac{d} 2}\rfloor=c+\lceil {\frac {d} 2}\rceil, & \hbox{ if } c+\lceil {\frac {d} 2}\rceil \text{is even and } d \text{ is even}.
  \end{array}
\right.
$$
So,
$$c+\lceil {\frac {d'} 2}\rceil\left\{
  \begin{array}{ll}
    is ~odd, & \hbox{ if } d \text{ is odd}; \\
    is ~even, & \hbox{ if } d \text{ is even}.
  \end{array}
\right.
$$
We get
$$ Y'=\chi(d ~is ~odd).$$
\end{proof}

\subsection{$\k=(a,c,e)$ with $a \leq c$}
The formula of $B(b,d)$ is
$$B(b,d)=\left\{
               \begin{array}{ll}
                 3b+d-a,    & \text{ if }   2(a-b) \leq d ; \\
                 2b+ \lceil {\frac{d} 2}\rceil, & \text{ if }  2(a-b) > d.
               \end{array}
             \right.
$$
Here, we give a map
$\varphi: \oD\mapsto \varphi(\oD)$ 
where $\varphi(\oD)$ is determined by its two values $(b', d')$,
i.e., $\varphi(b,d) = (b',d')$ which $(b', d')$ is defined as follows:
$$(b', d')=\left\{
               \begin{array}{ll}
                 (b, 3a-5b+c-d), & \text{ if } 2(a-b) \leq d   \text{ and } 3b+d-a \leq c ; \\
                 (\frac{a-b+c-d-X} 2, 2a-2b+X),    & \text{ if } 2(a-b) \leq d   \text{ and } 3b+d-a>c ; \\
                 (a-\lfloor {\frac{d} 2}\rfloor, 2d-2b+c-3\lceil {\frac{d} 2}\rceil)  & \text { if } 2(a-b) > d  \text{ and } 2b+ \lceil {\frac{d} 2}\rceil \leq c; \\
                 (a-b-d+ {\frac{c+\lceil {\frac d 2}\rceil-Y} 2}, 2\lfloor {\frac{d} 2}\rfloor+Y), & \text { if } 2(a-b) > d  \text{ and } 2b+ \lceil {\frac{d} 2}\rceil > c.
               \end{array}
             \right.
$$
where $X$ is $\chi((a-b+c-d)~ is ~odd)$ and $Y$ is $\chi((c+\lceil {\frac d 2}\rceil) ~is ~odd)$.

According to the formulas of $B(b,d)$ and $(b',d')$, we define the classification situation as follows:

\begin{description}
  \item[L1] $ 2(a-b) \leq d $,
  \item[L2] $ 2(a-b) > d$;
\end{description}
and
\begin{description}
  \item[L11] $ 2(a-b) \leq d   \text{ and } 3b+d-a \leq c$,
  \item[L12] $2(a-b) \leq d   \text{ and } 3b+d-a>c$,
  \item[L21] $2(a-b) > d  \text{ and } 2b+ \lceil {\frac{d} 2}\rceil \leq c $,
  \item[L22] $2(a-b) > d  \text{ and } 2b+ \lceil {\frac{d} 2}\rceil > c$.
\end{description}

\begin{prop}\label{p-leqinvolution}
the map $\varphi$ is an $\emph{involution}$ on $\CD_{\k}$
for $\k=(a,c,e)$ with $a \leq c$ and it interchanges $\area$ and $\bounce$.
\end{prop}


We prove Proposition $\ref{p-leqinvolution}$ in two steps as follows:
\begin{enumerate}
   \item[step 1:]
   \begin{enumerate}
  \item $D=\varphi(\oD)$ is a Dyck path, i.e., $D \in \CD_{\k}$, satisfying $a-b' \geq 0$ and $a-b'+c-d' \geq 0$;
  \item the map $\varphi$ is an $\emph{involution}$, i.e., $\varphi(b',d')=(b'',d'')=(b,d)$.
       Indeed, the map $\varphi$ is also an involution when restricted to condition L11, and the same situation holds for L22;
       It exchanges conditions L12 and L21. See Figure \ref{fig:leq}.
   \end{enumerate}
  \item[step 2:] The map $\varphi$ interchanges $\area$ and $\bounce$, i.e., $\area(D)=\bounce(\oD)$ and $\bounce(D)=\area(\oD)$.
\end{enumerate}
\begin{figure}[!ht]
  $$
 \hskip 0.05in\vcenter { \includegraphics[height= 2.06 in]{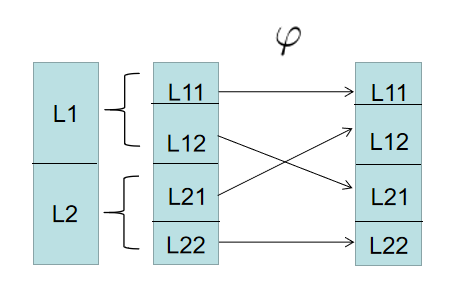}}
$$
\caption{ The $\emph{involution}$ $\varphi$ and its image.
\label{fig:leq}}
\end{figure}
\begin{proof}
Recall that $a \leq c$, $a-b \geq 0$ and $a-b+c-d \geq 0$.

\medskip
\begin{enumerate}
\item[step 1:]
Case 11: When condition L11 holds, we have $(b',d')=(b, 3a-5b+c-d)$.

We can obtain
\begin{enumerate}
  \item \begin{align*}
            a-b'&=a-b \geq 0; \\
            a-b'+c-d'&=a-b+c-(3a-5b+c-d)=d-2(a-b)+2b \geq 2b \geq 0.
        \end{align*}
  \item It is easy to check that the condition L11 for $(b',d')$ holds:
  $$2(a-b')-d'=2(a-b)-(3a-5b+c-d)=3b+d-a-c \leq 0$$ and $$3b'+d'-a=3b+(3a-5b+c-d)-a=2(a-b)-d+c \leq c.$$
So,
\begin{align*}
(b'', d'')=(b',3a-5b'+c-d')=(b,3a-5b+c-(3a-5b+c-d))=(b,d).
\end{align*}
\end{enumerate}
Case 12: When condition L12 holds, we have
$(b', d')=({\frac{a-b+c-d-X} 2}, 2a-2b+X).$

We can obtain
\begin{enumerate}
  \item \begin{align*}
        a-b'&=a- {\frac{a-b+c-d-X} 2} = \frac{a+b-c+d+X} 2; \\
        a-b'+c-d'&=a-{\frac{a-b+c-d-X} 2}+ c -(2a-2b+X)\\
                 &=\frac{-3a+5b+c+d-X} 2.
        \end{align*}
    Recall that condition L12: $2(a-b) \leq d$ and $3b+d-a>c$. Then $2b \geq 2a-d$, $3b>a+c-d$ and $a+b-c+d > 2a-2b \geq 0$.
    We have $5b>3a+c-2d$.
    Thus,
        \begin{align*}
        &a-b' = \frac{a+b-c+d+X} 2 > 0; \\
        &a-b'+c-d' =\frac{-3a+5b+c+d-X} 2 > \frac{2c-d-X} 2.
        \end{align*}
        By $a\leq c$ and $d\leq a+c \leq 2c$, we have $2c-d-X <0 $ only if $2c=d$ and $X=1$.
        But when $2c=d$, we have $a=c$, $b=0$ and $a-b+c-d=0$, i.e., $X=0$.
         This contradicts with $X=1$. So, $a-b'+c-d' > 0$.
  \item It is easy to check that the condition L21 for $(b',d')$ holds:
    \begin{align*}
    2(a-b')-d'&=2a-(a-b+c-d-X)-(2a-2b+X)\\
    &=2a+c-(a\!-\!b\!-\!d)-2a+2b\\
    &=3b+d-a-c > 0
    \end{align*}
  and
    \begin{align*}
    2b'+\lceil {\frac{d'} 2}\rceil &=a-b+c-d-X+a-b+X \\
    &=2a-2b-d+c \leq c.
    \end{align*}
So,
\begin{align*}
(b'', d'')&=(a-\lfloor {\frac{d'} 2}\rfloor, 2d'-2b'+c-3\lceil {\frac{d'} 2}\rceil)\\
&=(a-(a-b),4a-4b+2X-(a-b+c-d-X)+c-3(a-b+X)\\
&=(b,d).
\end{align*}
\end{enumerate}
Case 21: When condition L21 holds, we have $$(b', d')=(a-\lfloor {\frac{d} 2}\rfloor, 2d-2b+c-3\lceil {\frac{d} 2}\rceil).$$

We can obtain
\begin{enumerate}
  \item \begin{align*}
         a-b'&=a-(a-\lfloor {\frac{d} 2}\rfloor)=\lfloor {\frac{d} 2}\rfloor \geq 0; \\
         a-b'+c-d'&=\lfloor {\frac{d} 2}\rfloor+c-(2d-2b+c-3\lceil {\frac{d} 2}\rceil)=2b-d+2\lceil {\frac{d} 2}\rceil \geq 0.
         \end{align*}
  \item It is easy to check that the condition L12 for $(b',d')$ holds:
  \begin{align*}
                2(a-b')-d'&=2a-2b'-d' \\
                        &=2a-2(a-\lfloor {\frac{d} 2}\rfloor)-(2d-2b+c-3\lceil {\frac{d} 2}\rceil)) \\
                        &=2b+ \lceil {\frac{d} 2}\rceil-c \leq 0
            \end{align*}
  and
  \begin{align*}
               3b'+d'-a&=3(a-\lfloor {\frac{d} 2}\rfloor)+(2d-2b+c-3\lceil {\frac{d} 2}\rceil)-a\\
                       &=2(a-b)-d+c > c.
            \end{align*}
  So,
            $$(b'', d'')=( {\frac{a-b'+c-d'-X'} 2}, 2a-2b'+X')$$
        where $X'$ is $\chi(a-b'+c-d' ~is ~odd)$ which is equal to $\chi(d ~ is ~odd)$ by the formula $a-b'+c-d'=2b-d+2\lceil {\frac{d} 2}\rceil$ from $(a)$. Then we have
            \begin{align*}
                (b'', d'')&=(b+\lceil {\frac{d} 2}\rceil-{\frac{d} 2}-{\frac{X'} 2},2a-2(a-\lfloor {\frac{d} 2}\rfloor)+X')\\
                          &=(b+\lceil {\frac{d} 2}\rceil-{\frac{d+X'} 2}, 2\lfloor {\frac{d} 2}\rfloor+X')\\
                          &=(b,d).
            \end{align*}
\end{enumerate}
Case 22: When condition L22 holds, we have
$$(b',d')=(a-b-d+ {\frac{c+\lceil {\frac d 2}\rceil-Y} 2}, 2\lfloor {\frac{d} 2}\rfloor+Y).$$

We can obtain
\begin{enumerate}
  \item  \begin{align*}
                a-b'&=a-(a-b-d+ {\frac{c+\lceil {\frac d 2}\rceil-Y} 2})
                ={\frac{2b+d-c+\lfloor {\frac d 2}\rfloor+Y} 2}; \\
           a-b'+c-d'&={\frac{2b+2d-c-\lceil {\frac d 2}\rceil+Y} 2}+c-(2\lfloor {\frac{d} 2}\rfloor+Y)\\
                    &={\frac{2b+2d+c-\lceil {\frac d 2}\rceil+Y-4\lfloor {\frac{d} 2}\rfloor-2Y} 2}\\
                    &={\frac{2b+2d+c-\lceil {\frac d 2}\rceil-4\lfloor {\frac{d} 2}\rfloor-Y} 2}\\
                    &={\frac{2b+d+c-3\lfloor {\frac{d} 2}\rfloor-Y} 2}.
         \end{align*}
        Recall that condition L22:  $2(a-b) > d$ and $2b+\lceil {\frac d 2}\rceil > c$.
        We are in the case $a\leq c$. Then $2b+\lceil {\frac d 2}\rceil > c \geq a > {\frac d 2}+b$ and
        $2b+d-c+\lfloor {\frac d 2}\rfloor+Y > 0$.

        We know $$a-b' > 0$$ and
            $2b+d+c-3\lfloor {\frac{d} 2}\rfloor-Y >3b+{\frac{3d} 2}-3\lfloor {\frac{d} 2}\rfloor-Y > 0$, i.e.,
            $$a-b'+c-d' >0.$$
  \item It is easy to check that the condition L22 for $(b',d')$ holds:
   \begin{align*}
              2(a-b')-d'&=2a-2(a-b-d+ {\frac{c+\lceil {\frac d 2}\rceil-Y} 2})-(2\lfloor {\frac{d} 2}\rfloor+Y)\\
                        &=2b+2d-c-\lceil {\frac d 2}\rceil+Y-2\lfloor {\frac{d} 2}\rfloor - Y\\
                        &=2b+ \lceil {\frac{d} 2}\rceil-c >0
            \end{align*}
  and
  \begin{align*}
                2b'+ \lceil {\frac{d'} 2}\rceil&=2(a-b-d+ {\frac{c+\lceil {\frac d 2}\rceil-Y} 2})+\lfloor {\frac{d} 2}\rfloor+Y\\
                          &=2a-2b-2d+ c+\lceil {\frac d 2}\rceil-Y+\lfloor {\frac{d} 2}\rfloor+Y\\
                          &=2a-2b-d+c > c.
            \end{align*}
  So,
            \begin{align*}
                (b'', d'')&=(a-b'-d'+ {\frac{c+\lceil {\frac {d'} 2}\rceil-Y'} 2}, 2\lfloor {\frac{d'} 2}\rfloor+Y')\\
                &=(a-(a-b-d+ {\frac{c+\lceil {\frac d 2}\rceil-Y} 2})-(2\lfloor {\frac{d} 2}\rfloor+Y)+ {\frac{c+\lfloor {\frac{d} 2}\rfloor+Y-Y'} 2},2\lfloor {\frac{d} 2}\rfloor+Y')\\
                &=(b + {\frac {d-2\lfloor{\frac{d} 2}\rfloor-Y'} 2},2\lfloor {\frac{d} 2}\rfloor+Y')
            \end{align*}
        where $Y'$ is $\chi((c+\lceil {\frac {d'} 2}\rceil) ~is ~odd)$.

        Recall that $Y$ is $\chi((c+\lceil {\frac {d} 2}\rceil) ~is ~odd)$ and $d'=2\lfloor {\frac{d} 2}\rfloor+Y$.
        So, by Lemma \ref{lem-ymula}, we have
        $$(b'', d'')=(b + {\frac {d-2\lfloor{\frac{d} 2}\rfloor-Y'} 2},2\lfloor {\frac{d} 2}\rfloor+Y')=(b,d).$$
\end{enumerate}

\item[step 2:]
Case 11:  When condition L11 holds, we have $\area(\oD)=2a-2b+c-d$ and $\bounce(\oD)=3b+d-a$.
    We calculate
    \begin{align*}
      \area(D)&=2a-2b'+c-d'=2a-2b+c-(3a-5b+c-d)\\
      &=3b+d-a=\bounce(\oD);\\
    \bounce(D)&=3b'+d'-a=3b+(3a-5b+c-d)-a \\
           &=2a-2b+c-d=\area(\oD).
    \end{align*}
Case 12:  When condition L12 holds, we have $\area(\oD)=2a-2b+c-d$ and $\bounce(\oD)=3b+d-a$.
We calculate
\begin{align*}
\area(D)&=2a-2b'+c-d' = 2a-(a-b+c-d-X) +c-(2a-2b+X)\\
        &=2a+c-(a\!-\!b\!+\!c\!-\!d)-2a+2b\\
        &=3b+d-a=\bounce(\oD);\\
\bounce(D)&=2b'+\lceil {\frac{d'} 2}\rceil = a-b+c-d-X+a-b+X \\
          &=2a-2b+c-d=\area(\oD).
\end{align*}
Case 21:  When condition L21 holds, we have $\area(\oD)=2a-2b+c-d$ and $\bounce(\oD)=2b+ \lceil {\frac{d} 2}\rceil$.
        We calculate
            \begin{align*}
            \area(D)&=2a-2b'+c-d' \\
                        &=2a-2(a-\lfloor {\frac{d} 2}\rfloor)+c-(2d-2b+c-3\lceil {\frac{d} 2}\rceil)) \\
                        &=2b+ \lceil {\frac{d} 2}\rceil=\bounce(\oD);\\
            \bounce(D)&=3b'+d'-a \\
                      &=3(a-\lfloor {\frac{d} 2}\rfloor)+(2d-2b+c-3\lceil {\frac{d} 2}\rceil)-a\\
                      &=2a-2b+c-d=\area(\oD).
            \end{align*}
Case 22:  When condition L12 holds, we have $\area(\oD)=2a-2b+c-d$ and $\bounce(\oD)=2b+ \lceil {\frac{d} 2}\rceil$.
        We calculate
            \begin{align*}
                \area(D)&=2a-2b'+c-d' \\
                        &=2a-2(a-b-d+ {\frac{c+\lceil {\frac d 2}\rceil-Y} 2})+c-(2\lfloor {\frac{d} 2}\rfloor+Y)\\
                        &=2b+2d-c-\lceil {\frac d 2}\rceil+Y+c-2\lfloor {\frac{d} 2}\rfloor - Y\\
                        &=2b+ \lceil {\frac{d} 2}\rceil=\bounce(\oD);\\
              \bounce(D)&=2b'+ \lceil {\frac{d'} 2}\rceil \\
                        &=2(a-b-d+ {\frac{c+\lceil {\frac d 2}\rceil-Y} 2})+\lfloor {\frac{d} 2}\rfloor+Y\\
                        &=2a-2b-2d+ c+\lceil {\frac d 2}\rceil-Y+\lfloor {\frac{d} 2}\rfloor+Y\\
                        &=2a-2b+c-d=\area(\oD).
            \end{align*}
\end{enumerate}
\end{proof}
\subsection{$\k=(a,c,e)$ with $a > c$}
The formula of $B(b,d)$ is
$$B(b,d)=\left\{
               \begin{array}{ll}
                 2b-c+d,    & \text{ if } a-b > c  \text{ and }  2c \leq d;\\
                 3b+d-a,    & \text{ if } a-b \leq c  \text{ and }  2(a-b) \leq d ; \\
                 2b+ \lceil {\frac{d} 2}\rceil, & \text{ if } \min(2(a-b),2c)> d.
               \end{array}
             \right.
$$

Here, we give a map
$\psi: D\mapsto \psi(D) \in \CD_{\k}$, where $\psi(D)$ is determined by its two values $(b', d')$,
i.e., $\psi(b,d) = (b',d')$ which $(b',d')$ is defined as follows:
\begin{align*}
(b', d')&=\left\{
               \begin{array}{ll}
                 (a-b+c-d, d),    & \text{ if } a-b > c  \text{ and }  2c \leq d \\
                                  &\text{ except } b = 0 \text { and } d = 2c; \\
                 (a-c, d),        & \text{ if }  b=0 \text{ and }  d = 2c;\\
                 (0,d),           &\text{ if } b = a-c \text { and } d = 2(a-b); \\
                 ( {\frac{a-b+c-d-X} 2}, 2a-2b+X),    & \text{ if } a-b \leq c  \text{ and }  2(a-b) \leq d \\
                                                      &\text{ except } b = a-c \text { and } d = 2(a-b); \\
                 (a-\lfloor {\frac{d} 2}\rfloor, 2d-2b+c-3\lceil {\frac{d} 2}\rceil),  & \text { if }\min(2(a-b),2c)> d \\
                                                                                       &\text{  and } 2b+\lceil {\frac{d} 2}\rceil \leq c;  \\
                 (a-b-d- {\frac{c+\lceil {\frac d 2}\rceil-Y} 2}, 2\lfloor {\frac{d} 2}\rfloor+Y), & \text { if }\min(2(a-b),2c)> d \\
                                                                                                   &\text{ and } 2b+\lceil {\frac{d} 2}\rceil > c.
               \end{array}
             \right.
\end{align*}
where $X$ is $\chi((a-b+c-d) ~is ~odd)$ and $Y$ is $\chi((c+\lceil {\frac d 2}\rceil) ~is ~odd)$.

According to the formulas of $B(b,d)$ and $(b',d')$, we define the classification situation as follows:

\begin{description}
  \item[G1] $ a-b > c $ and $2c \leq d$,
  \item[G2] $ a-b \leq c$ and $2(a-b) \leq d$,
  \item[G3] $\text{min}(2(a-b),2c) > d$;
\end{description}
and
\begin{description}
  \item[G11] $ a-b > c $ and $2c \leq d$ except $b=0$ and $2c=d$(G1/G12),
  \item[G12] $b=0$ and $2c=d$,
  \item[G21] $a-b=c$ and $2(a-b)=d$,
  \item[G22] $ a-b \leq c$ and $2(a-b) \leq d$ except $a-b=c$ and $2(a-b)=d$(G2/G21),
  \item[G31] $\text{min}(2(a-b),2c) > d  \text{ and } 2b+ \lceil {\frac{d} 2}\rceil \leq c $,
  \item[G32] $\text{min}(2(a-b),2c) > d  \text{ and } 2b+ \lceil {\frac{d} 2}\rceil > c$.
\end{description}
\begin{prop}\label{p-geqinvolution}
The map $\psi$ is an $\emph{involution}$ on $\CD_{\k}$
for $\k=(a,c,e)$ with $a > c$ and it interchanges $\area$ and $\bounce$.
\end{prop}

\begin{figure}[!ht]
  $$
 \hskip 0.05in\vcenter { \includegraphics[height= 2.06 in]{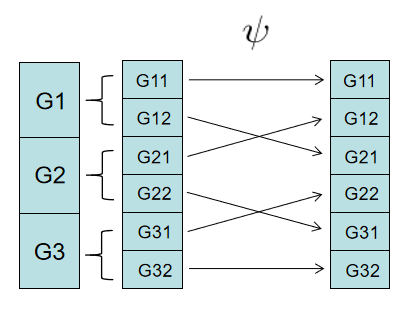}}
$$
\caption{The $\emph{involution}$ $\psi$ and its image.
\label{fig:geq}}
\end{figure}
\begin{proof}
This proof is similar to the case $a \leq c$, but we note that  condition G12 contains only one Dyck path, so does condition G21. They exchange with each other under the map $\psi$.

We only outline the two steps of the proof as follows and omit the details.
\begin{enumerate}
  \item[step 1:]
  \begin{enumerate}
  \item For a given Dyck path $\oD\in \CD_{\k}$, $D=\psi(\oD)$ is also in $\CD_{\k}$, i.e., satisfying $a-b' \geq 0$ and $a-b'+c-d' \geq 0$;
  \item The map $\psi$ is an $\emph{involution}$, i.e., $\psi(b',d')=(b'',d'')=(b,d)$.
  Indeed, the map $\psi$ is also an involution when restricted to condition G11, and the same situation holds for  G32;
  It exchanges the two Dyck paths under the conditions G12 and G21;
  It exchanges conditions G22 and G31. See Figure \ref{fig:geq}.

  \end{enumerate}
  \item[step 2:] The map $\psi$ interchanges $\area$ and $\bounce$, i.e., $\area(D)=\bounce(\oD)$ and $\bounce(D)=\area(\oD)$.
\end{enumerate}
\end{proof}

\section{An Algebraic Proof of $C_{k^4}(q,t)=C_{k^4}(t,q)$ \label{s-Mac-proof2}}
In this section, we use the same idea in section \ref{s-Mac-proof} to prove $C_{k^4}(q,t)=C_{k^4}(t,q)$.
A $k^4$-Dyck path $D$ can be uniquely determined by the parameters $a,b,c$ as follows. It starts with
a red arrow $S^{k}$ (i.e. up step $(1,k)$) followed by $a$ blue arrows $W$ (down step $(1,-1)$), then a red arrow $S^{k}$ followed by $b$ blue arrows $W$, and a red arrow $S^{k}$ followed by $c$ blue arrows $W$, finally a red arrow $S^{k}$ followed by $(4k-a-b-c)$ blue arrows $W$.
The parameters $a,b,c$ satisfy the conditions $0\le a \le k$ , $0\le b \le 2k-a $ and $0\le c \le 3k-a-b$.
The red ranks $(r_1=0,r_2,r_3,r_4)$ of $D$ can be written by the parameters $a,b,c$ as follows.
$$r_2 = k-a, \ r_3=2k-a-b, \  r_4=3k-a-b-c.$$
Then, the area of $D$ is simply
$\area(D)=r_2+r_3+r_4=6k-3a-2b-c$
and the bounce of $D$ is given in \cite{Niu} by
\begin{align}\label{bounceformula4k}
&\bounce(D)=   \nonumber \\
 &\left\{
         \begin{array}{ll}
                 6a+3b+c-4k, & \text{ if } b \geq 2k-2a \text{ and } c \geq 4k-2a-2b; \\
                 5a+2b+\lceil {\frac{c} 2}\rceil-2k, & \text{ if } b \geq 2k-2a \text{ and } c < 4k-2a-2b; \\
                 4a+2b+c-2k, & \text{ if }  b < 2k-2a , b\text{ is even and } c \geq 3k-a-\frac{3b}{2}; \\
                 2a+\frac {b}{2}+k+\lceil{\frac {3a+\frac{3b}{2}+c-3k}{2}}\rceil, & \text{ if }  b < 2k-2a ,  b\text{ is even and } \\
                                                                              & \ \ \    3k-3a-\frac{3b}{2} \leq c < 3k-a-\frac{3b}{2}; \\
                 3a+b+ \lceil {\frac{c} 3}\rceil, & \text{ if }  b < 2k-2a , b\text{ is even and } c < 3k-3a-3s; \\
                 4a+2b+c-2k+1 , & \text{ if }  b < 2k-2a ,  b\text{ is odd and } c \geq 3k-a-\frac{3(b+1)}{2}+1; \\
                 2a+\frac{b+1}{2}+k &
                                      \\
                        \ \ \ \ \ \ \ \   +\lceil{\frac {3a+\frac {3(b+1)}{2}+c-3k-1}{2}}\rceil , & \text{ if }  b < 2k-2a ,  b\text{ is odd and } \\
                                                                                      & \ \ \ 3k-3a-\frac{3(b+1)}{2}+1 \leq c < 3k-a-\frac{3(b+1)}{2}+1; \\
                 3a+b+1+\lceil {\frac{c-1} 3}\rceil  , & \text{ if }  b < 2k-2a ,  b\text{ is odd and } c < 3k-3a-\frac{3(b+1)}{2}+1.
         \end{array}
         \right.
\end{align}

We can construct the generating function with respect to $k,a,b,c$:
\begin{align*}
H(x,y_2,y_3,y_4,q,t) &= \sum_{k\ge 0}x^{k} \sum_{D\in \CD_{K}} q^{\area(D)} t^{\bounce(D)} y_2^{a}y_3^{b}y_4^{c}  \\
                       &= \sum_{k\ge 0}x^{k} \sum_{0\le a\le k, \ 0\le b\le 2k-a, \ 0\le c\le 3k-a-b} q^{6k-3a-2b-c} t^{\bounce(D)} y_2^{a}y_3^{b}y_4^{c},
\end{align*}
where we allowed $k=0$, which does not affect the computation.

Then it is sufficient to prove the $q,t$ symmetry of
$$H(x,1,1,1,q,t)= \sum_{k\ge 0}x^{k} C_{k^4}(q,t).$$

\subsection{Crude generating function}
Due to the piecewise linearity of $\bounce(D)$ in formula (\ref{bounceformula4k}),
we divide the generating function $H(x,y_2,y_3,y_4,q,t)$ into three parts as follows.
\begin{description}
  \item[Part 1] $b\geq 2k - 2a$;
  \item[Part 2] $b < 2k - 2a$ and $b$ is even;
  \item[Part 3] $b < 2k - 2a$ and $b$ is odd.
\end{description}

There are two cases for Part 1 and 3 cases for Part 2 and Part 3 each, so it is convenient to write the generating function $H(x,y_2,y_3,y_4,q,t)=H=H_1+H_2+H_3=H_{11}+H_{12}+H_{21}+H_{22}+H_{23}+H_{31}+H_{32}+H_{33}$.

Part 1: $b\geq 2k - 2a$, i.e., $2a+b-2k \ge 0$. By a similar computation for section \ref{s-Mac-proof}, the formulas $H_{11}$ and $H_{12}$ are given as follows.

   Case 1: $c \geq 4k-2a-2b$, i.e., $2a+2b+c-4k \geq 0$.
   Then, we have $$\bounce(D) =6a+3b+c-4k.$$
and
\begin{align*}
H_{11}=&\sum_{k\ge 0} x^{k} \sum_{0\le a \le k, \ 0\le b\le 2k-a, \  0\le c \le 3k-a-b,  \atop  2a+b-2k \ge 0, \  2a+2b+c-4k \geq 0} q^{6k-3a-2b-c} t^{6a+3b+c-4k} y_2^{a}y_3^{b}y_4^{c}  \\
=&\sum_{k,a,b,c \ge 0} \Oge \left(x^kq^{6k-3a-2b-c} t^{6a+3b+c-4k} y_2^{a}y_3^{b}y_4^{c} \lambda_1^{k-a}\lambda_2^{2k-a-b}\lambda_3^{3k-a-b-c}
\lambda_4^{2a+b-2k}\lambda_5^{2a+2b+c-4k}  \right) \\
=& \Oge \frac {1} {(1-\frac{xq^6\lambda_1\lambda_2^2\lambda_3^3}{t^4\lambda_4^2\lambda_5^4})
(1-\frac{t^6y_2\lambda_4^2\lambda_5^2}{q^3\lambda_1\lambda_2\lambda_3})
(1-\frac{t^3y_3\lambda_4\lambda_5^2}{q^2\lambda_2\lambda_3})
(1-\frac{ty_4\lambda_5}{q\lambda_3})}.
\end{align*}

  Case 2: $c < 4k-2a-2b$, i.e., $4k-2a-2b-c-1 \geq 0$. Then, we have
  \begin{align*}
  \bounce(D)= 5a+2b+\lceil \frac {c} 2 \rceil - 2k,
  \end{align*}
and
\begin{align*}
H_{12}=\Oge \mathop{\Omega}_{\mu=} \frac {1+\mu}{\lambda_5(1-\frac{xq^6\lambda_1\lambda_2^2\lambda_3^3\lambda_5^4}{t^2\lambda_4^2})
(1-\frac{t^5y_2\lambda_4^2}{q^3\lambda_1\lambda_2\lambda_3\lambda_5^2})
(1-\frac{t^2y_3\lambda_4}{q^2\lambda_2\lambda_3\lambda_5^2})
(1-\frac{y_4\mu}{q\lambda_3\lambda_5})
(1-\frac{t}{\mu^2})}.
\end{align*}

Part 2: $b<2k-2a$ (i.e.,$2k-2a-b-1 \ge 0$) and $b$ is even. Let $b=2s$, then $2k-2a-2s-1 \ge 0$ and $\area(D)= 6k-3a-2b-c= 6k-3a-4s-c$.

Case 1: $c\geq 3k-a-\frac{3b}{2}$, i.e., $a+3s+c-3k\geq 0$.
Then, we have $$\bounce(D)=4a+2b+c-2k=4a+4s+c-2k.$$
We can write the generating function $H_{21}$ as follows.
\begin{align*}
&H_{21}= \sum_{k\ge 0}x^{k} \sum_{0\le a\le k, \ 0\le 2s\le 2k-a, \ 0\le c\le 3k-a-2s,\atop 2k-2a-2s-1 \ge 0,\  a+3s+c- 3k \geq 0} q^{6k-3a-4s-c} t^{4a+4s+c-2k} y_2^{a}y_3^{2s}y_4^{c}\\
=&\sum_{k,a,s,c\ge 0} \Oge \left(x^{k}q^{6k-3a-4s-c} t^{4a+4s+c-2k} y_2^{a}y_3^{2s}y_4^{c} \lambda_1^{k-a}\lambda_2^{2k-a-2s}\lambda_3^{3k-a-2s-c}
\lambda_4^{2k-2a-2s-1}\lambda_5^{a+3s+c- 3k} \right) \\
=&
\Oge \frac{1}{\lambda_4(1-\frac{xq^6\lambda_1\lambda_2^2\lambda_3^3\lambda_4^2}{t^2\lambda_5^3})
(1-\frac{t^4y_2\lambda_5}{q^3\lambda_1\lambda_2\lambda_3\lambda_4^2})
(1-\frac{t^4y_3^2\lambda_5^3}{q^4\lambda_2^2\lambda_3^2\lambda_4^2})
(1-\frac{ty_4\lambda_5}{q\lambda_3})}.
\end{align*}

Case 2: $3k-3a-\frac{3b}{2} \leq c < 3k-a-\frac{3b}{2}$, i.e., $3a+3s+c-3k \geq 0$ and $3k-a-3s-c-1 \geq 0$.
Then, we have $$\bounce(D)=2a+\frac {b}{2}+k+\lceil{\frac {3a+\frac{3b}{2}+c-3k}{2}}\rceil=2a+s+k+\lceil{\frac {3a+3s+c-3k}{2}}\rceil.$$

Let $p=\lceil{\frac {3a+3s+c-3k}{2}}\rceil$.
By Lemma \ref{skillformula}, we can write the generating function $H_{22}$ as follows.
\begin{align*}
H_{22}&=\Oge \mathop{\Omega}_{\mu=} \frac {1+\mu}{\lambda_4\lambda_6
(1-\frac{xq^6t\lambda_1\lambda_2^2\lambda_3^3\lambda_4^2\lambda_6^3}{\mu^3\lambda_5^3})
    (1-\frac{t^2y_2\mu^3\lambda_5^3}{q^3\lambda_1\lambda_2\lambda_3\lambda_4^2\lambda_6})
    (1-\frac{ty_3^2\mu^3\lambda_5^3}{q^4\lambda_2^2\lambda_3^2\lambda_4^2\lambda_6^3})
    (1-\frac{y_4\mu\lambda_5}{q\lambda_3\lambda_6})
    (1-\frac{t}{\mu^2})};
\end{align*}

Case 3: $c < 3k-3a-\frac{3b}{2}$, i.e., $3k-3a-3s-c-1 \geq 0$.
Then, we have $$\bounce(D)=3a+b+\lceil{\frac {c}{3}}\rceil=3a+2s+\lceil{\frac {c}{3}}\rceil.$$

Let $p=\lceil{\frac {c}{3}}\rceil$.
By Lemma \ref{skillformula}, we can write the generating function $H_{23}$ as follows.
\begin{align*}
H_{23}&=\Oge \mathop{\Omega}_{\mu=} \frac {1+\mu+\mu^2}{\lambda_4\lambda_5(1-xq^6\lambda_1\lambda_2^2\lambda_3^3\lambda_4^2\lambda_5^3)
   (1-\frac{t^3y_2}{q^3\lambda_1\lambda_2\lambda_3\lambda_4^2\lambda_5^3})(1-\frac{t^2y_3^2}{q^4\lambda_2^2\lambda_3^2\lambda_4^2\lambda_5^3})
   (1-\frac{y_4\mu}{q\lambda_3\lambda_5})(1-\frac{t}{\mu^3})};
\end{align*}

Part 3: $b<2k-2a$ (i.e.,$2k-2a-b-1 \ge 0$) and $b$ is odd. Let $b=2s+1$, then $2k-2a-2s-2 \ge 0$ and $\area(D)= 6k-3a-2b-c= 6k-3a-4s-c-2$.
By a similar computation for Part 2, the formulas $H_{31}, H_{32}, H_{33}$ are given as follows.
\begin{align*}
   H_{31}&=\Oge  \frac {t^3y_3\lambda_5^2}{q^2\lambda_2\lambda_3\lambda_4^2(1-\frac{xq^6\lambda_1\lambda_2^2\lambda_3^3\lambda_4^2}{t^2\lambda_5^3})
   (1-\frac{t^4y_2\lambda_5}{q^3\lambda_1\lambda_2\lambda_3\lambda_4^2})(1-\frac{t^4y_3^2\lambda_5^3}{q^4\lambda_2^2\lambda_3^2\lambda_4^2})
   (1-\frac{ty_4\lambda_5}{q\lambda_3})};\\
    H_{32}&=\Oge \mathop{\Omega}_{\mu=} \frac{t(1+\mu)\mu^2y_3\lambda_5^2}{q^2\lambda_2\lambda_3\lambda_4^2\lambda_6^3
    (1-\frac{xq^6t\lambda_1\lambda_2^2\lambda_3^3\lambda_4^2\lambda_6^3}{\mu^3\lambda_5^3})
    (1-\frac{t^2\mu^3y_2\lambda_5^3}{q^3\lambda_1\lambda_2\lambda_3\lambda_4^2\lambda_6})} \\
    &\ \ \ \ \ \ \ \ \ \ \ \ \ \ \  \ \ \ \ \ \ \ \ \ \ \ \ \ \ \ \ \ \ \ \ \ \ \ \ \ \ \  \times \frac {1}
    {(1-\frac{t\mu^3y_3^2\lambda_5^3}{q^4\lambda_2^2\lambda_3^2\lambda_4^2\lambda_6^3})
    (1-\frac{\mu y_4\lambda_5}{q\lambda_3\lambda_6})(1-\frac{t}{u^2})};\\
   H_{33}&=\Oge \mathop{\Omega}_{\mu=} \frac {t^2(1+\mu+\mu^2)y_3}{q^2\mu\lambda_2\lambda_3\lambda_4^2\lambda_5^3
   (1-xq^6\lambda_1\lambda_2^2\lambda_3^3\lambda_4^2\lambda_5^3)
   (1-\frac{t^3y_2}{q^3\lambda_1\lambda_2\lambda_3\lambda_4^2\lambda_5^3})} \\
   &\ \ \ \ \ \ \ \ \ \ \ \ \ \ \  \ \ \ \ \ \ \ \ \ \ \ \ \ \ \ \ \ \ \ \ \ \ \ \ \ \ \ \times 
   \frac {1}{(1-\frac{t^2y_3^2}{q^4\lambda_2^2\lambda_3^2\lambda_4^2\lambda_5^3})
   (1-\frac{\mu y_4}{q\lambda_3\lambda_5})(1-\frac{t}{\mu^3})}.
\end{align*}

\subsection{Obtain the generating function}
By using the maple package Ell2, we obtain:

\begin{align*}
H_{11}&=\frac{1-{q}^{2}{t}^{8}{x}^{2}y_{{2}}{y_{{3}}}^{3}y_{{4}}} {\left({1-{q}^{2}{t}^{2}x{y_{{3}}}^{2}}\right)\left({1-q{t}^{3}x{y_{{3}}}^{2}y_{{4}}}\right)\left({1-q{t}^{4}xy_{{2}}{y_{{4}}}^{2}}\right)\left({1-q{t}^{5}xy_{{2}}y_{{3}}}\right)\left({1-{t}^{6}xy_{{2}}y_{{3}}y_{{4}}}\right)}
,\\
H_{12}&=\frac{y_{{2}} q ^{2}t ^{3}x \left({q+ty_{{4}}}\right)} {\left({1-{q}^{2}{t}^{2}x{y_{{3}}}^{2}}\right)\left({1-{q}^{3}{t}^{3}xy_{{2}}}\right)\left({1-q{t}^{5}xy_{{2}}y_{{3}}}\right)\left({1-q{t}^{4}xy_{{2}}{y_{{4}}}^{2}}\right)},\\
H_{21}&=\frac{q ^{3}t x y_{{4}} ^{3}} {\left({1-{q}^{2}{t}^{2}x{y_{{3}}}^{2}}\right)\left({1-{q}^{3}tx{y_{{4}}}^{3}}\right)\left({1-q{t}^{3}x{y_{{3}}}^{2}y_{{4}}}\right)\left({1-q{t}^{4}xy_{{2}}{y_{{4}}}^{2}}\right)},\\
H_{22}&=\frac{y_{{2}} y_{{4}} ^{3}x ^{2}t ^{4}q ^{5}\left({q+ty_{{4}}}\right)} {\left({1-{q}^{2}{t}^{2}x{y_{{3}}}^{2}}\right)\left({1-{q}^{3}tx{y_{{4}}}^{3}}\right)\left({1-{q}^{3}{t}^{3}xy_{{2}}}\right)\left({1-q{t}^{4}xy_{{2}}{y_{{4}}}^{2}}\right)},\\
H_{23}&=\frac{q ^{4}x \left({{q}^{2}+t{y_{{4}}}^{2}+qty_{{4}}}\right)} {\left({1-{q}^{2}{t}^{2}x{y_{{3}}}^{2}}\right)\left({1-{q}^{6}x}\right)\left({1-{q}^{3}tx{y_{{4}}}^{3}}\right)\left({1-{q}^{3}{t}^{3}xy_{{2}}}\right)},\\
H_{31}&=\frac{q ^{2}t ^{2}x y_{{4}} y_{{3}} \left({q+ty_{{4}}-{q}^{2}{t}^{3}x{y_{{3}}}^{2}y_{{4}}}\right)} {\left({1-{q}^{2}{t}^{2}x{y_{{3}}}^{2}}\right)\left({1-{q}^{3}tx{y_{{4}}}^{3}}\right)\left({1-q{t}^{3}x{y_{{3}}}^{2}y_{{4}}}\right)\left({1-q{t}^{4}xy_{{2}}{y_{{4}}}^{2}}\right)}
,\\
H_{32}&=\frac{y_{{2}} y_{{3}} y_{{4}} x ^{2}t ^{5}q ^{5}\left({q+ty_{{4}}}\right)} {\left({1-{q}^{2}{t}^{2}x{y_{{3}}}^{2}}\right)\left({1-{q}^{3}tx{y_{{4}}}^{3}}\right)\left({1-{q}^{3}{t}^{3}xy_{{2}}}\right)\left({1-q{t}^{4}xy_{{2}}{y_{{4}}}^{2}}\right)},\\
H_{33}&=\frac{y_{{3}} t ^{2}x q ^{4}\left({1+{q}^{5}xy_{{4}}+{q}^{4}tx{y_{{4}}}^{2}}\right)} {\left({1-{q}^{2}{t}^{2}x{y_{{3}}}^{2}}\right)\left({1-{q}^{6}x}\right)\left({1-{q}^{3}tx{y_{{4}}}^{3}}\right)\left({1-{q}^{3}{t}^{3}xy_{{2}}}\right)}.
\end{align*}

By adding the above eight formulas and setting $y_2=y_3=y_4=1$, we obtain
\begin{align}\label{qtsym-formual2}
H(x,1,1,1,q,t)
=\frac {N}{\left({1-{q}^{3}tx}\right)\left({1-{q}^{2}{t}^{2}x}\right)\left({1-q{t}^{3}x}\right)\left({1-{q}^{6}x}\right)\left({1-{t}^{6}x}\right)}
\end{align}
where
\begin{align*}
N=1 &+ \left( {q}^{5}t+{q}^{4}{t}^{2}+{q}^{3}{t}^{3}+{q}^{2}{t}^{4}+q{t}^{
5}+{q}^{4}t+{q}^{3}{t}^{2}+{q}^{2}{t}^{3}+q{t}^{4} \right) x \\
&+\left( {q}^{6}{t}^{5}+{q}^{5}{t}^{6}-{q}^{7}{t}^{3}-{q}^{6}{t}^{4}-{q}^{5}{t}^{
5}-{q}^{4}{t}^{6}-{t}^{7}{q}^{3}-{q}^{5}{t}^{4}-{q}^{4}{t}^{5}
 \right) {x}^{2} \\
&- \left( {q}^{8}{t}^{8}+{q}^{9}{t}^{6}+{q}^{8}{t}^{7}+
{q}^{7}{t}^{8}+{q}^{6}{t}^{9} \right){x}^{3}.
\end{align*}
Then $C_{k^4}(q,t)=C_{k^4}(t,q)$ clearly follows from \eqref{qtsym-formual2}.

\section{Concluding remark}
In this paper we give two proofs of the $q,t$-symmetry of the generalized $q,t$-Catalan number
$C_{\k}(q,t)$ for $\k$ of length $3$.

For $\k$ of length $n\ge 4$, we conjectured the $q,t$-symmetry of $C_{\lambda}(q,t)$ when
$\lambda=((a+1)^s,a^{n-s})$. MacMahon's partition analysis technique may be used to attack the
conjecture provided that we can find a piecewise linear formula of the bounce statistic.
Such a formula seems DO exist, but becomes complicated even for
$n=4$.


%
%
%


\begin{thebibliography}{99}

\bibitem{Ome-package}
G. E. Andrews, P. Paule, and A. Riese, MacMahon's partition analysis: the Omega package, European J. Combin. 22(2001).


\bibitem{qt-Catalan}
A. Garsia and M. Haiman, A remarkable $q, t$-Catalan sequence and $q$-Lagrange inversion, J. Algebraic
Combinatorics 5 (1996), 191--244.

\bibitem{Garsia-Haglund-proof-positivity}
A. Garsia and J. Haglund. "A proof of the q, t-Catalan positivity conjecture". Discrete Math. 256 (2002), pp. 677--717.


\bibitem{dinv-area}
A. Garsia and G. Xin, Dinv and Area, Electron. J. Combin., 24 (1) (2017), P1.64.


\bibitem{Gorsky-Mazin}
E. Gorsky and M. Mazin, Compactified Jacobians and $q,t$-Catalan Numbers I, J. Combin. Theory Ser. A, 120 (2013), 49--63.


\bibitem{Haglund-bounce}
J. Haglund, \newblock Conjectured Statistics for the $q, t$-Catalan numbers, Advances in Mathematics 175 (2003), 319--334.

\bibitem{Hag-book08}
J. Haglund, The $q, t$-Catalan numbers and the space of diagonal harmonics, with an appendix on the
combinatorics of Macdonald polynomials, AMS University Lecture Series, 2008.

\bibitem{Higher-qt-Catalan}
Nicholas A. Loehr. Conjectured Statistics for the Higher $q,t$-Catalan Sequences[J]. Electronic Journal of Combinatorics, 2005, 12(1):318--344.

\bibitem{Niu}
M. Niu. The $q,t$-symmetry of the generalized $q,t$-Catalan number $C_{\lambda=(k^s,(k-1)^{n-s})}(q,t)$ when $n=4$, M.S. Thesis. Capital Normal University, 2022. (Written in Chinese)


\bibitem{fastCT}
G. Xin, A fast algorithm for MacMahon's partition analysis, Electron. J. Combin. 11 (2004),R58, 20 pp. (electronic).

\bibitem{Xin-Zhang}
G. Xin and Y. Zhang, On the Sweep Map for $\vec{k}$-Dyck Paths, Electron. J. Combin., 26 (3) (2019), P3.63.

\bibitem{Xin-Zhang2}
G. Xin and Y. Zhang, Dinv, Area and Bounce for $\vec{k}$-Dyck Paths, arXiv:2011.04927.

\end{thebibliography}
\end{document}